\documentclass[hidelinks,12pt,a4paper,reqno]{amsart}
\usepackage{graphicx}
\usepackage[all]{xy}
\usepackage{amssymb}
\usepackage{amsmath}
\usepackage[utf8]{inputenc}
\usepackage{amsthm}
\usepackage{multicol}

\newtheorem{proposition}{Proposition}

\newtheorem{definition}{Definition}
\newtheorem{theorem}{Theorem}

\newcommand{\C}{\mathbf{C}}

\newcommand{\Mon}{\mathbf{Mon}}

\newcommand{\Set}{\mathbf{Set}}
\newcommand{\Mag}{\mathbf{Mag}}
\newcommand{\Act}{\mathbf{Act}}
\newcommand{\Psb}{\mathbf{Psb}}

\newcommand{\map}{\mathrm{map}}

\title[Pointed semibiproducts of monoids]{Pointed semibiproducts of monoids}

\author{Nelson Martins-Ferreira}
\address[Nelson Martins-Ferreira]{Instituto Politécnico de Leiria, Leiria, Portugal}
\thanks{ }
\email{martins.ferreira@ipleiria.pt}

\begin{document}

\begin{abstract}
    Semibiproducts of monoids are introduced here as a common generalization to biproducts (of abelian groups) and to semidirect products (of groups) for exploring a wide class of monoid extensions. More generally, abstract semi\-bi\-products exist in any concrete category over sets in which map addition is meaningful thus reinterpreting Mac Lane's relative biproducts. In the pointed case they give rise to a special class of extensions called semibiproduct extensions. Not all monoid extensions are  semibiproduct extensions but all group extensions are. A categorical equivalence is established between the category of pointed semibiproducts of monoids and the category of pointed monoid action systems, a new category of actions that emerges from the equivalence. The main difference to classical extension theory is that semibiproduct extensions are treated in the same way as split extensions, even though the section map may fail to be a homomorphism. A list with all the 14 semibiproduct extensions of 2-element monoids is provided.
\end{abstract}



\maketitle

\section{Introduction}

Biproducts were introduced by Mac Lane in his book \emph{Homology} to study split extensions in the context of abelian categories. Semidirect products are appropriate to study group split extensions. Although these concepts have been thoroughly developed over the last decades in the context of protomodular and semi-abelian categories \cite{BB,Bourn,BournJanelidze,semiabelian}, the notion of \emph{relative biproduct} introduced by Mac Lane to study  relative split extensions seems to have been forgotten (see \cite{Homology}, p.~263). On the other hand, much work has been done in  extending the tools and techniques from  groups \cite{MacLane2} to monoids \cite{DB.NMF.AM.MS.13, NMF14,DB.NMF.AM.MS.16,Faul,Fleischer,Ganci,Leech,NMF et all,Wells} and even more general settings \cite{GranJanelidzeSobral}. However, as it has been observed several times, it is not a straightforward task to take a well-known result in the category of groups (or any other semi-abelian category) and materialize it in the category of monoids not to mention more general situations. We will argue that a convenient reformulation of relative biproduct (called \emph{semibiproduct}) can be used to study group and monoid extensions as a unified frame of work. Even though semidirect products are suitable to describe all group split extensions they fail to capture those group extensions that do not split.
The key observation to semibiproducts in reinterpreting relative biproducts (see \cite{Homology}, diagram (5.2), p.~263 and compare with diagram $(\ref{diag: biproduct in a Mag-category})$ in Definition  \ref{def: semibiproduct}) is that although an extension may fail to split as a monoid extension or as a group extension, it necessarily splits as an extension of pointed sets.

The main result (Theorem \ref{thm: equivalence}) establishes an equivalence of categories between pointed semibiproducts of monoids (Definition \ref{def: pointed semibiproduct of monoids}) and pointed monoid action systems (Definition \ref{def: pseudo-action}). The 14 classes of non-isomorphic pointed semibiproducts of 2-element monoids are listed in Section \ref{sec: eg}. We start with some motivation in Section \ref{sec: motivation}, introduce $\Mag$-categories and semibiproducts in Section \ref{Sec: Mag-categories}, restrict to the pointed case in Section~\ref{sec: stability} while studying some stability properties and pointing out some differences and similarities between groups, monoids and unitary magmas. From Section \ref{Sec: Sbp} on we work towards the main result and restrict our attention to monoids.

\section{Motivation}\label{sec: motivation}

It is well known that a split extension of groups
\[\xymatrix{X\ar[r]^k & A \ar[r]^{p} & B,} \]
with a specified section $s\colon{B\to A}$, can be completed into a  diagram of the form
\[\xymatrix{X\ar@<-.5ex>[r]_{k} & A\ar@<-.5ex>@{..>}[l]_{q}\ar@<.5ex>[r]^{p} & B \ar@{->}@<.5ex>[l]^{s},}\]
in which $q\colon{A\to X}$ is the map uniquely determined by the formula $kq(a)=a-sp(a)$, $a\in A$. Furthermore, the information needed to reconstruct the split extension as a semidirect product is encoded in the map $\varphi\colon{B\times X\to X}$, uniquely determined as $\varphi(b,x)=q(s(b)+k(x))$. If writing the element $\varphi(b,x)\in X$ as $b\cdot x$ we see that $k(b\cdot x)$ is equal to $s(b)+k(x)-s(b)$ and that the group $A$ is recovered as the semidirect product $X\rtimes_{\varphi}B$. In the event that the section $s$, while being a zero-preserving map, is not necessarily a group homomorphism then the classical treatment of group extensions prescribes a different procedure (see e.g.\ \cite{Northcott}, p.~238). However, the results obtained here suggest that non-split extensions may be treated similarly to split extensions, and moreover the same approach is carried straightforwardly into the context of monoids.
Indeed, when $s$ is not a homomorphism, in addition to the map $\varphi$, we get a map $\gamma\colon{B\times B\to X}$, determined as $\gamma(b,b')=q(s(b)+s(b'))$ and the group $A$ is recovered as the set $X\times B$ with group operation
\[(x,b)+(x',b')=(x+\varphi(b,x')+\gamma(b,b'),b+b')\]
defined for every $x,x'\in X$ and $b,b'\in B$. Note that $X$ needs not be commutative. However, instead of simply saying that $\varphi$ is an action and that $\gamma$ is a factor system, we have to consider two maps $\varphi$ and $\gamma$ which in conjunction turn the set $X\times B$ into a group with a prescribed operation (Section \ref{Sec: Act}). This is precisely what we call a semibiproduct of groups. Observe that when $s$ is a homomorphism it reduces to the usual notion of semidirect product.

Almost every step in the treatment of groups is carried over into the context of monoids. However, while in groups all extensions are obtained as semibiproducts, in monoids we have to restrict our attention to those for which there exists a section $s$ and a retraction $q$ satisfying the condition $a=kq(a)+sp(a)$ for all $a\in A$ (Section \ref{Sec: Sbp}). Consequently, in addition to the maps $\varphi$ and $\gamma$ obtained as in groups, a new map $\rho\colon{X\times B\to X}$, determined as $\rho(x,b)=q(k(x)+s(b))$ needs to be taken into consideration. Hence, the monoid $A$ is recovered as the set $\{(x,b)\in X\times B\mid\rho(x,b)=x\}$ with operation
\begin{equation}\label{eq: operation}
    (x,b)+(x',b')=(\rho(x+\varphi(b,x')+\gamma(b,b'),b+b'),b+b')
\end{equation}
which is defined for every $x,x'\in X$ and $b,b'\in B$.


\section{$\Mag$-categories and semibiproducts}\label{Sec: Mag-categories}

Let $\Mag$ denote the category of magmas and magma homomorphisms and let $U\colon{\Mag\to \Set}$ be the forgetful functor into the category of sets and maps. By a $\Mag$-category we mean a category $\C$ together with a bifunctor $\map\colon{\C^{\text{op}}\times \C\to \Mag}$ and a natural inclusion $\varepsilon\colon{\hom_{\C}\to U\circ \map}$. If $\C$ is a concrete category over sets in which a meaningful map addition is available then a bifunctor $\map$ is obtained as follows. For every pair of objects $(A,B)$ in $\C$, $\map(A,B)$ is the magma of underlying maps from object $A$ to object $B$ equipped with component-wise addition. In particular $\map(A,B)$ contains $\hom_{\C}(A,B)$ as a subset since $$\varepsilon_{A,B}\colon{\hom_{\C}(A,B)\to U(\map(A,B))}$$ is required to be a natural inclusion. As expected the category $\Mag$ is a $\Mag$-category with $\map(A,B)$ the magma of all maps from $U(A)$ to $U(B)$. If $f$ is a magma homomorphism from $A$ to $B$ then $\varepsilon_{A,B}(f)$ is nothing but $f$ considered as a map between the underlying sets of $A$ and $B$. In the same way the categories of groups, abelian groups, monoids and commutative monoids are $\Mag$-categories. However, there is a significant distinction between the Ab-category of abelian groups, the linear category  of commutative monoids and the Mag-categories of groups or monoids. If $A$ is an object in an Ab-category then  $\hom(A,A)$ is  a ring. If $A$ is an object in  a linear category then $\hom(A,A)$  is a semiring. In contrast, if $A$ is a group (or a monoid) then $\hom(A,A)$ is a subset of the near-ring $\map(A,A)$. 

\begin{definition}\label{def: semibiproduct}
    Let $(\C,\map,\varepsilon)$ be a $\Mag$-category. A \emph{semibiproduct} is a tuple $(X,A,B,p,k,q,s)$ represented as a diagram of the shape
    \begin{equation}
        \label{diag: biproduct in a Mag-category}
        \xymatrix{X\ar@<-.5ex>[r]_{k} & A\ar@<-.5ex>@{..>}[l]_{q}\ar@<.5ex>[r]^{p} & B \ar@{..>}@<.5ex>[l]^{s}}
    \end{equation}
    in which $p\colon{A\to B}$ and $k\colon{X\to A}$ are morphisms in $\C$,  whereas $q\in \map(A,X)$ and $s\in \map(B,A)$. Furthermore, the following conditions are satisfied:
    \begin{eqnarray}
        ps={1_B}\label{eq: biproduct1}\\ qk={1_{X}},\label{eq: biproduct2}\\
        kq+sp={1_A}.\label{eq: biproduct3}
    \end{eqnarray}
\end{definition}

There is an obvious abuse of notation in the previous conditions. This is justified because we will be mostly concerned with the case in which $\C$ is the category of monoids and $\map(A,B)$ is the set of zero-preserving maps. In rigorous terms, condition $ps=1_B$ should have been written as $\map(1_B,p)(s)=\varepsilon_{B,B}(1_B)$ whereas condition $qk=1_X$
should have been written as $\map(k,1_X)(q)=\varepsilon_{X,X}(1_X)$. In the same way the condition $\map(1_A,k)(q)+\map(p,1_A)(s)=\varepsilon_{A,A}(1_A)$ should have been written in the place of $kq+sp=1_A$.
For the moment we will not develop this concept further but rather concentrate our attention on  the concrete cases of  groups  and monoids.

\section{Stability properties of pointed semibiproducts}\label{sec: stability}

From now on the category $\C$ is assumed to be either the category of groups or the category of monoids (occasionally we will refer to the category of unitary magmas) and $\map(A,B)$ is the magma of zero-preserving maps with component-wise addition.
In each case the category $\C$ is pointed and the composition of maps is well defined. It will be convenient to consider \emph{pointed semibiproducts} by requiring two further conditions, namely
\begin{eqnarray}
    pk=0_{X,B},\quad qs=0_{B,X}.
\end{eqnarray}
However, as it is well known, in the case of groups this distinction is irrelevant.

\begin{proposition} Every semibiproduct of groups is pointed.
\end{proposition}
\begin{proof}
    We have $pk=p1_Ak=p(kq+sp)k=pkqk+pspk=pk+pk$. And $s=1_As=(kq+sp)s=kqs+sps=kqs+s$. Hence we may conclude $pk=0$ and $kqs=0$. Since $k$ is a monomorphism $qs=0$.
\end{proof}

The previous proof also shows that a semibiproduct of monoids is pointed as soon as the monoid $A$ admits right cancellation. Clearly, this is not a general fact.

\begin{proposition}
    Let $A$ be a monoid. The tuple $(A,A,A,1_A,1_A,1_A,1_A)$ is a semibiproduct of monoids if and only if $A$ is an idempotent monoid.
\end{proposition}
\begin{proof}
    Condition $(\ref{eq: biproduct3})$ in this case becomes $a=a+a$ for all $a\in A$.
\end{proof}

Every pointed semibiproduct of monoids has an underlying exact sequence.

\begin{proposition}\label{thm:kernel of p}\label{thm:cokernel of k}
    Let $(X,A,B,p,k,q,s)$ be a pointed semibiproduct of monoids. The sequence
    \[\xymatrix{X\ar[r]^k & A \ar[r]^{p} & B} \] is an exact sequence.
\end{proposition}
\begin{proof}
    Let $f\colon{Z\to A}$ be a morphism such that $pf=0$. Then the map $\bar{f}=qf$ is a homomorphism
    \begin{eqnarray*}
        qf(z+z')&=&q(fz+fz')=q(kqf(z)+spf(z)+kqf(z')+spf(z'))\\
        &=& q(kqf(z)+0+kqf(z')+0)\\
        &=& qk(qf(z)+qf(z'))=qf(z)+qf(z')
    \end{eqnarray*}
    and it is unique with the property $k\bar{f}=f$. Indeed, if $k\bar{f}=f$ then $qk\bar{f}=qf$ and hence $\bar{f}=qf$. This means that $k$ is the kernel of $p$.

    Let $g\colon{A\to Y}$ be a morphism and suppose that $gk=0$. It follows that $g=gsp$,
    \begin{equation*}
        g=g1_A=g(kq+sp)=gkq+gsp=0+gsp=gsp,
    \end{equation*}
    and consequently the map $\bar{g}=gs$ is a homomorphism, indeed
    \begin{eqnarray*}
        gs(b)+gs(b')=g(sb+sb')=gsp(sb+sb')=gs(b+b').
    \end{eqnarray*}
    The fact that $\bar{g}=gs$ is the unique morphism with the property $\bar{g}p=g$ follows from $\bar{g}ps=gs$ which is the same as $\bar{g}=gs$. Hence $p$ is the cokernel of $k$ and the sequence is exact.
\end{proof}

The following results show that pointed semibiproducts are stable under pullback and in particular split semibiproducts of monoids are stable under composition.

\begin{proposition}\label{thm:stable under pullback} Pointed semibiproducts of monoids are stable under pullback.
\end{proposition}
\begin{proof}
    Let $(X,A,B,p,k,q,s)$ be a pointed semibiproduct of monoids displayed as the bottom row in the following diagram which is obtained by taking the pullback of $p$ along an arbitrary morphism $h\colon{C\to B}$, with induced morphism $\langle k,0 \rangle$ and map $\langle sh,1 \rangle$,
    \begin{eqnarray}
        \vcenter{\xymatrix{X\ar@<-.5ex>[r]_(.35){\langle k,0 \rangle}\ar@{=}[d]_{} & A\times_B C\ar@{->}@<0ex>[d]^{\pi_1}\ar@<-.5ex>@{..>}[l]_(.6){q\pi_1}\ar@<.5ex>[r]^(.6){\pi_2} & C\ar@{->}[d]^{h} \ar@{->}@<.5ex>@{..>}[l]^(.35){\langle sh,1\rangle}\\
        X\ar@<-.5ex>[r]_{k} & A \ar@<-.5ex>@{..>}[l]_{q}\ar@<.5ex>[r]^{p} & B \ar@{->}@<.5ex>@{..>}[l]^{s}.}}
    \end{eqnarray}
    We have to show that the top row is a pointed semibiproduct of monoids. By construction we have $\pi_2\langle sh,1\rangle=1_C$, $\pi_2\langle k,0\rangle=0$, $q\pi_1\langle sh,1\rangle=qsh=0$, $q\pi_1\langle k,0\rangle=qk=1_X$. It remains to prove the identity
    \[(a,c)=(kq(a),0)+(sh(c),c)=(kq(a)+sh(c),c)\]
    for every $a\in A$ and $c\in C$ with $p(a)=h(c)$, which follows from $a=kq(a)+sp(a)=kq(a)+sh(c)$.
\end{proof}

The previous results are stated at the level of monoids but are easily extended to unitary magmas. The  particular case of semidirect products has  been considered in \cite{GranJanelidzeSobral} and the notion of composable pair of pointed semibiproducts is borrowed from  there.
We say that a pointed semibiproduct $(X,A,B,p,k,q,s)$ can be composed with a pointed semibiproduct $(C,B,D,p',k',q',s')$  if the tuple $$(A\times_B C,A,D,p'p,\pi_1,q'',ss'),$$ in which $q''$ is such that $\pi_1q''=kq+sk'q'p$ and $\pi_2q''=q'p$, is a  pointed semibiproduct.
Note that in the case of groups the map $q$  is uniquely determined as $q(a)=a-sp(a)$ for all $a\in A$. However this is not the case for monoids nor for unitary magmas.

\begin{proposition}
    A pointed semibiproduct of monoids $(X,A,B,p,k,q,s)$ can be composed  with $(C,B,D,p',k',q',s')$, another pointed semibiproduct of monoids,  if and only if the map $s$ is equal to the map $sk'q'+ss'p'$.
\end{proposition}
\begin{proof}
    Let us observe that the tuple $(A\times_B C,A,D,p'p,\pi_1,q'',ss')$ is a pointed semibiproduct if and only if $\pi_1q''+ss'p'p=1_A$.
    Indeed, the kernel of the composite $p'p$ is obtained by taking the pullback of $p$ along $k'$, the kernel of $p'$, as illustrated
    \begin{eqnarray}
        \vcenter{\xymatrix{ & A\times_B C\ar@{->}@<0ex>[d]_{\pi_1}\ar@<.5ex>[r]^(.6){\pi_2} & C\ar@{->}[d]_{k'} \ar@{->}@<.5ex>@{..>}[l]^(.35){\langle sk',1\rangle}\\
        X\ar@<-.5ex>[r]_{k} & A \ar@<-.5ex>@{..>}[l]_{q}\ar@<.5ex>[r]^{p} \ar[d]_{p'p} & B \ar@{->}@<.5ex>@{..>}[l]^{s} \ar[d]_{p'}\\
        & D\ar@{=}[r] & D.}}
    \end{eqnarray}
    In order to obtain a pointed semibiproduct we complete de diagram with a map $q''$ such that $\pi_1q''=kq+sk'q'p$ and $\pi_2q''=q'p$ as illustrated
    \begin{eqnarray}
        \vcenter{\xymatrix{ & A\times_B C\ar@{->}@<-0.5ex>[d]_{\pi_1}\ar@<.5ex>[r]^(.6){\pi_2} & C\ar@{->}[d]_{k'} \ar@{->}@<.5ex>@{..>}[l]^(.35){\langle sk',1\rangle}\\
        X\ar@<-.5ex>[r]_{k} & A \ar@<-.5ex>@{..>}[l]_{q} \ar@<-.5ex>@{..>}[u]_{q''} \ar@<.5ex>[r]^{p} \ar@<-.5ex>[d]_{p'p}  & B \ar@{->}@<.5ex>@{..>}[l]^{s} \ar[d]_{p'}\\
        & D \ar@<-.5ex>@{..>}[u]_{ss'} \ar@{=}[r] & D.}}
    \end{eqnarray}
    The map $q''$ is well defined, $p(kq+sk'q'p)=pkq+psk'q'p=k'q'p$. Moreover, $p'pss'=1_D$, $p'p\pi_1=p'k'\pi_2=0$, $q''ss'=0$ and we observe
    \begin{eqnarray*}
        q''\pi_1 &=&\langle kq+sk'q'p,q'p \rangle \pi_1\\
        &=&\langle kq\pi_1+sk'q'p\pi_1,q'p\pi_1 \rangle \\
        &=&\langle kq\pi_1+sk'q'k'\pi_2,q'k'\pi_2 \rangle \\
        &=&\langle kq\pi_1+sp\pi_1,\pi_2 \rangle \\
        &=&\langle \pi_1,\pi_2 \rangle =1_{A\times_B C}.
    \end{eqnarray*}
    It remains to analyse the condition $\pi_1q''+ss'p'p=1_A$. If $s=sk'q'+ss'p'$ then we have $\pi_1q''+ss'p'p=kq+sk'q'p+ss'p'p$ and hence $kq+sp=1_A$. Conversely, having $\pi_1q''+ss'p'p=1_A$ we get $kq+sk'q'p+ss'p'p=1_A$ and $kqs+sk'q'ps+ss'p'ps=s$ so $sk'q'+ss'p'=s$.
\end{proof}

Note that associativity is used to convert $(kq+sk'q'p)+ss'p'p$ into $kq+(sk'q'p+ss'p'p)$. Moreover, if the map $s$ is a homomorphism then condition $s=sk'q'+ss'p'$ is trivial.  A pointed semibiproduct $(X,A,B,p,k,q,s)$ in which the map $s$ is a homomorphism is called a pointed split semibiproduct.
This means that pointed split semibiproducts of monoids are stable under composition.

\section{The category of pointed semibiproducts of monoids}\label{Sec: Sbp}

The purpose of this section is to introduce the category of pointed semibiproducts of monoids, denoted $\Psb$.

\begin{definition}\label{def: pointed semibiproduct of monoids}
    A \emph{pointed semibiproduct of monoids} consists of
    a tuple $(X,A,B,p,k,q,s)$ that can also be represented as a diagram of the shape
    \begin{equation}
        \label{diag: biproduct}
        \xymatrix{X\ar@<-.5ex>[r]_{k} & A\ar@<-.5ex>@{..>}[l]_{q}\ar@<.5ex>[r]^{p} & B \ar@{..>}@<.5ex>[l]^{s}}
    \end{equation}
    in which $X$, $A$ and $B$ are monoids (not necessarily commutative), $p$, $k$, are monoid homomorphisms, while  $q$ and $s$ are zero-preserving maps. Moreover, the following conditions are satisfied:
    \begin{eqnarray}
        ps&=&1_B\\
        qk&=&1_X\\
        kq+sp&=&1_A\\
        pk&=&0_{X,B}\\
        qs&=&0_{B,X}.
    \end{eqnarray}
\end{definition}

A morphism in $\Psb$, from the object $(X,A,B,p,k,q,s)$ to the object $(X',A',B',p',k',q',s')$,  is a triple $(f_1,f_2,f_3)$, displayed as
\begin{equation}\label{diag:morphism of semi-biproduct}
    \vcenter{\xymatrix{X\ar@<-.5ex>[r]_{k}\ar@{->}[d]_{f_1} & A\ar@{->}@<0ex>[d]^{f_2}\ar@<-.5ex>@{..>}[l]_{q}\ar@<.5ex>[r]^{p} & B\ar@{->}[d]^{f_3} \ar@{->}@<.5ex>@{..>}[l]^{s}\\
    X'\ar@<-.5ex>[r]_{k'} & A' \ar@<-.5ex>@{..>}[l]_{q'}\ar@<.5ex>[r]^{p'} & B' \ar@{->}@<.5ex>@{..>}[l]^{s'}}}
\end{equation}
in which $f_1$, $f_2$ and $f_3$ are monoid homomorphisms and moreover the following conditions are satisfied: $f_2k=k'f_1$, $p'f_2=f_3p$, $f_2s=s'f_3$, $q'f_2=f_1q$.

\begin{theorem}\label{thm: a+a' and associativity}
    Let $(X,A,B,p,k,q,s)$ be a pointed semibiproduct of monoids. For every $a,a'\in A$ the element $a+a'\in A$ can be written in terms of $q(a)$, $q(a')$, $p(a)$ and $p(a')$ as
    \begin{equation}\label{eq: a+a'}
        k(q(a)+q(sp(a)+ kq(a'))+q(sp(a)+ sp(a')))+s(p(a)+p(a')).
    \end{equation}
\end{theorem}
\begin{proof}
    We observe:
    \begin{eqnarray*}
        a+a'&=& kqa+(spa+kqa')+spa'  \qquad(kq+sp=1)\\
        &=& kqa+kq(spa+kqa')+sp(spa+kqa')+spa'  \\
        &=& kqa+kq(spa+kqa')+spa+spa'  \qquad( ps=1,pk=0)\\
        &=&  kqa+kq(spa+kqa')+kq(spa+spa')+sp(spa+spa')\\
        &=&  kqa+kq(spa+kqa')+kq(spa+spa')+s(pa+pa')\\
        &=& k(qa+q(spa+kqa')+q(spa+spa'))+sp(a+a').
    \end{eqnarray*}
\end{proof}

The previous result suggests a transport of structure from the monoid $A$ into the set  $X\times B$ as motivated with formula $(\ref{eq: operation})$ in Section \ref{sec: motivation}. However, as we will see, in order to keep an isomorphism with $A$ we need to restrict the set $X\times B$ to those pairs $(x,b)$ for which there exists $a\in A$ such that $x=q(a)$ and $b=p(a)$.

\section{The category of pointed monoid action systems}\label{Sec: Act}

The purpose of this section is to introduce the category of pointed monoid action systems, which will be denoted as $\Act$. This category is obtained by requiring the existence of a categorical equivalence between $\Act$ and $\Psb$ (see Theorem \ref{thm: equivalence}).

\begin{definition}\label{def: pseudo-action}
    A \emph{pointed monoid action system} is a five-tuple $$(X,B,\rho,\varphi,\gamma)$$ in which $X$ and $B$ are monoids, $\rho\colon{X\times B\to X}$, $\varphi\colon{B\times X\to X}$, $\gamma\colon{B\times B\to X}$ are maps such that the following conditions are satisfied for every $x\in X$ and $b,b'\in B$:
    \begin{eqnarray}
        \rho(x,0)=x,\quad \rho(0,b)=0\label{eq:act01}\\
        \varphi(0,x)=x,\quad \varphi(b,0)=0\label{eq:act02}\\
        \gamma(b,0)=0=\gamma(0,b)\label{eq:act03}\\
        \rho(x,b)=\rho(\rho(x,b),b)\label{eq:act04}\\
        \varphi(b,x)=\rho(\varphi(b,x),b)\label{eq:act05}\\
        \gamma(b,b')=\rho(\gamma(b,b'),b+b')\label{eq:act06}
    \end{eqnarray}
    and moreover the following condition holds for every $x,x',x''\in X$ and $b,b',b''\in B$,
    \begin{eqnarray}\label{eq:act07}
        \rho(\rho(x+\varphi(b,x')+\gamma(b,b'),b+b')+\varphi(b+b',x'')+\gamma(b+b',b''),b''')=\nonumber\\
        =\rho(x+\varphi(b,\rho(x'+\varphi(b', x'') + \gamma(b', b''),{b'+b''}))+\gamma(b, b'+b''),{b'''})\quad
    \end{eqnarray}
    where $b'''=b+b'+b''$.
\end{definition}

A morphism of pointed monoid action systems, say from $(X,B,\rho,\varphi,\gamma)$ to $(X',B',\rho',\varphi',\gamma')$ is a pair $(f,g)$ of monoid homomorphisms, with $f\colon{X\to X'}$ and $g\colon{B\to B'}$ such that for every $x\in X$ and $b,b'\in B$
\begin{eqnarray}
    f(\rho(x,b))&=&\rho'(f(x),{g(b)}),\label{eq:act08}\\
    f(\varphi(b, x))&=&\varphi'(g(b), {f(x)}),\label{eq:act09}\\
    f(\gamma(b, b'))&=&\gamma'(g(b), g(b')).\label{eq:act10}
\end{eqnarray}

\begin{theorem}\label{thm: functor R syntehic construction}
    There exists a functor $R\colon{\Act\to \Mon}$ such that for every morphism in $\Act$, say $(f,g)\colon{(X,B,\rho,\varphi,\gamma)\to (X',B',\rho',\varphi',\gamma')}$, the diagram
    \begin{equation}\label{diag:semi-biproduct with R}
        \xymatrix{X\ar@<-.5ex>[r]_(.3){\langle 1,0\rangle}\ar@{->}[d]_{f} & R(X,B,\rho,\varphi,\gamma)\ar@{->}@<.0ex>[d]^{R(f,g)}\ar@<-.5ex>@{..>}[l]_(.7){\pi_X}\ar@<.5ex>[r]^(.7){\pi_B} & B\ar@{->}[d]^{g} \ar@{->}@<.5ex>@{..>}[l]^(.3){\langle 0,1\rangle}\\
        X'\ar@<-.5ex>[r]_(.3){\langle 1,0 \rangle} & R{(X',B',{\rho',\varphi',\gamma'})} \ar@<-.5ex>@{..>}[l]_(.7){\pi_{X}}\ar@<.5ex>[r]^(.7){\pi_B} & B \ar@{->}@<.5ex>@{..>}[l]^(.3){{\langle 0 ,1 \rangle}}}
    \end{equation}
    is a morphism in $\Psb$.
\end{theorem}

The functor $R$ realizes a pointed monoid action system  $(X,B,\rho,\varphi,\gamma)$ as a synthetic semibiproduct diagram
\begin{equation}\label{diag:synthetic semi-biproduct}
    \xymatrix{X\ar@<-.5ex>[r]_(.5){\langle 1,0\rangle} & R\ar@<-.5ex>@{..>}[l]_(.5){\pi_X}\ar@<.5ex>[r]^(.5){\pi_B} & B \ar@{..>}@<.5ex>[l]^(.5){\langle 0,1\rangle}}
\end{equation}
in which $R=R(X,B,{\rho,\varphi,\gamma})=\{(x,b)\in X\times B\mid \rho(x,b)=x\}$ is equipped with the binary synthetic operation
\begin{equation}\label{eq: semibiproduct sunthetic operation}
    (x,b)+(x',b')=(\rho(x+\varphi(b,x')+\gamma(b,b'),b+b'),b+b')
\end{equation}
which is well defined for every $x,x'\in X$ and $b,b'\in B$ due to condition $(\ref{eq:act04})$ and is associative due to condition $(\ref{eq:act07})$. It is clear that $\pi_B$ is a monoid homomorphism and due to conditions $(\ref{eq:act01})$--$(\ref{eq:act03})$ we see that the maps $\langle 1,0 \rangle $ and $\langle 0 ,1\rangle $ are well defined and moreover $\langle 1,0\rangle $ is a monoid homomorphism. Finally, we observe that a pair $(x,b)\in X\times B$ is in $R$ if and only if $(x,b)=(x,0)+(0,b)$. Further details  can be found in the preprint \cite{NMF.20a-of}.

\section{The equivalence}

In order to establish a categorical equivalence between $\Act$ and $\Psb$ we need a procedure to associate a pointed monoid action system to every pointed semibiproduct of monoids in a functorial manner.

\begin{theorem}\label{thm: pseudo-actions}
    Let $(X,A,B,p,k,q,s)$ be an object in $\Psb$. The system $(X,B,\rho,\varphi,\gamma)$ with
    \begin{eqnarray}
        \rho(x,b)=q(k(x)+s(b))\\
        \varphi(b,x)=q(s(b)+k(x))\\
        \gamma(b,b')=q(s(b)+s(b'))
    \end{eqnarray}
    is an object in $\Act$.
    Moreover, if $(f_1,f_2,f_3)$ is a morphism in $\Psb$ then $(f_1,f_3)$ is a morphism in $\Act$.
\end{theorem}

\begin{proof}
    To see that the system $(X,B,\rho,\varphi,\gamma)$ is a well defined object in $\Act$ we recall that $q$ and $s$ are zero-preserving maps and hence conditions $(\ref{eq:act01})$--$(\ref{eq:act03})$ are satisfied. Conditions $(\ref{eq:act04})$--$(\ref{eq:act06})$ are obtained by applying the map $q$ to both sides of equations
    \begin{eqnarray*}
        k(x)+s(b)=kq(k(x)+s(b))+s(b)\\
        s(b)+k(x)=kq(s(b)+k(x))+s(b)\\
        s(b)+s(b')=kq(s(b)+s(b'))+s(b+b')
    \end{eqnarray*}
    which hold because $(X,A,B,p,k,q,s)$ is a pointed semibiproduct of monoids. Condition $(\ref{eq:act07})$ follows from Theorem \ref{thm: a+a' and associativity} with $a=k(x)+s(b)$, $a'=k(x')+s(b')+k(x'')+s(b'')$ on the one hand whereas on the other hand $a=k(x)+s(b)+k(x')+s(b')$, $a'=k(x'')+s(b'')$. Moreover, the pair $(f_1,f_3)$ is a morphism of actions as soon as the triple $(f_1,f_2,f_3)$ is a morphism of semibiproducts, indeed we have
    \begin{eqnarray*}
        f_1(\rho(x,b))&=&f_1q(k(x)+s(b))=q'f_2(k(x)+s(b))\\
        &=&q'(k'f_1(x)+s'f_3(b)))=\rho'(f_1(x),f_3(b))
    \end{eqnarray*}
    and similarly for $\varphi$ and $\gamma$ thus proving conditions $(\ref{eq:act08})$--$(\ref{eq:act10})$.
\end{proof}

The previous result describes a functor from the category of pointed semibipro\-ducts of monoids into the category of pointed monoid action systems, let us denote it by $P\colon{\Psb\to\Act}$. The synthetic construction of Theorem \ref{thm: functor R syntehic construction} produces a functor in the other direction, let us denote it $Q\colon{\Act\to\Psb}$. We will see that $PQ=1$ whereas  $QP\cong 1$.

\begin{theorem}\label{thm: equivalence}
    The categories $\Psb$ and $\Act$ are equivalent.
\end{theorem}

\begin{proof}
    Theorem \ref{thm: pseudo-actions} tells us that $P(X,A,B,p,k,q,s)=(X,B,\rho,\varphi,\gamma)$ and $P(f_1,f_2,f_3)=(f_1,f_3)$ is a functor from $\Psb$ to $\Act$ whereas  Theorem \ref{thm: functor R syntehic construction} gives a functor $Q$ in the other direction. It is clear that $Q(X,B,\rho,\varphi,\gamma)=(X,R,B,\pi_B,\langle 1,0\rangle,\pi_X,\langle 0,1\rangle)$ is the synthetic realization displayed in $(\ref{diag:synthetic semi-biproduct})$ and hence it is a pointed semibiproduct. Moreover  $Q(f,g)=(f,R(f,g),g)$ with $R(f,g)$ illustrated as in  $(\ref{diag:semi-biproduct with R})$ and defined as $R(f,g)(x,b)=(f(x),g(b))$ is clearly a morphism of semibiproducts.

    We observe that $PQ(X,B,\rho,\varphi,\gamma)=(X,B,\rho,\varphi,\gamma)$ due to conditions $(\ref{eq:act05})$ and $(\ref{eq:act06})$. This proves $PQ=1$, in order to prove $QP\cong 1$ we need to specify natural isomorphisms $\alpha$ and $\beta$ as illustrated
    \begin{equation}
        \vcenter{\xymatrix{A\ar@<.5ex>[r]^(.3){\alpha_A}\ar[d]_{f_2} & RP(X,A,B,p,k,q,s)\ar@<.5ex>[l]^(.7){\beta_A}\ar[d]^{R(f_1,f_3)}\\
                A' \ar@<.5ex>[r]^(.3){\alpha_{A'}} & RP(X',A',B',p',k',q',s')\ar@<.5ex>[l]^(.7){\beta_{A'}} }}
    \end{equation}
    and show that they are compatible with diagrams $(\ref{diag:morphism of semi-biproduct})$ and $(\ref{diag:semi-biproduct with R})$. Indeed it is a routine calculation to check that $\alpha(a)=(q(a),p(a))$ and $\beta(x,b)=k(x)+s(b)$ are well defined natural isomorphisms  compatible with semibiproducts. Further details can be found in the preprint \cite{NMF.20a-of}.
\end{proof}

\section{Examples}\label{sec: eg}

Several examples can be found in the preprint \cite{NMF.20a-of}. Here we list all the possible pointed semibiproducts of monoids $(X,A,B,p,k,q,s)$ in which $X$ and $B$ are monoids with two elements. This particular case is interesting because it gives a simple list with all the possible components of an action system $(X,B,\rho,\varphi,\gamma)$. The equivalence of Theorem \ref{thm: equivalence} then gives us an easy way of checking all the possibilities. Let us denote by $M$ and $G$ the two monoids with two elements, $M$ being the idempotent monoid while $G$ being the group, both expressed in terms of multiplication tables as
\[M=\begin{pmatrix}
        1 & 2 \\
        2 & 2
    \end{pmatrix},
    \quad
    B=\begin{pmatrix}
        1 & 2 \\
        2 & 1
    \end{pmatrix}.\]
Note that we are using multiplicative notation so that $2\cdot 2=2$ in $M$, whereas in $G$ we have $2\cdot 2=1$. Due to restrictions $(\ref{eq:act01})$--$(\ref{eq:act03})$ we have the following two possibilities for each component $\rho$, $\varphi$ and $\gamma$:
\[\rho_0=\begin{pmatrix}
        1 & 1 \\
        2 & 2
    \end{pmatrix}
    ,
    \quad
    \rho_1=\begin{pmatrix}
        1 & 1 \\
        2 & 1
    \end{pmatrix},\]
\[\varphi_0=\begin{pmatrix}
        1 & 2 \\
        1 & 2
    \end{pmatrix},
    \quad
    \varphi_1=\begin{pmatrix}
        1 & 2 \\
        1 & 1
    \end{pmatrix},\]
\[\gamma_0=\begin{pmatrix}
        1 & 1 \\
        1 & 1
    \end{pmatrix},
    \quad
    \gamma_1=\begin{pmatrix}
        1 & 1 \\
        1 & 2
    \end{pmatrix}.\]
The following list shows all the possible 14 cases of pointed semibiproducts of monoids $(X,A,B,p,q,k,q,s)$ in which $X$ and $B$ are either $M$ or $G$ via the equivalence of Theorem \ref{thm: equivalence}.
\begin{multicols}{2}
    
\begin{enumerate}
    \item $(G,G,\rho_0,\varphi_0,\gamma_0)$
    \item $(G,G,\rho_0,\varphi_0,\gamma_1)$
    \item $(G,M,\rho_0,\varphi_0,\gamma_0)$
    \item $(G,M,\rho_0,\varphi_0,\gamma_1)$
    \item $(G,M,\rho_0,\varphi_1,\gamma_0)$
    \item $(G,M,\rho_1,\varphi_1,\gamma_0)$
    \item $(M,G,\rho_0,\varphi_0,\gamma_0)$
    \item $(M,G,\rho_0,\varphi_0,\gamma_1)$
    \item $(M,G,\rho_1,\varphi_1,\gamma_1)$
    \item $(M,M,\rho_0,\varphi_0,\gamma_0)$
    \item $(M,M,\rho_0,\varphi_0,\gamma_1)$
    \item $(M,M,\rho_0,\varphi_1,\gamma_0)$
    \item $(M,M,\rho_0,\varphi_1,\gamma_1)$
    \item $(M,M,\rho_1,\varphi_1,\gamma_0)$
\end{enumerate}

\end{multicols}
Note that the cases with $\gamma_0$ correspond to split extensions while the cases with $\rho_0$ correspond to Schreier extensions. The cases with $\rho_1$ correspond to $R=\{(1,1),(1,2),(2,1)\}$ since $(2,2)$ fails to be in $R$ because $\rho_1(2,2)=1\neq 2$. If interpreting $\varphi$ as an action then the map $\varphi_0$ is the trivial action whereas $\varphi_1$ is a non-trivial action.

\section{Conclusion}

A new tool has been introduced for the study of monoid extensions from which a new notion of action has emerged in order to establish the categorical equivalence of Theorem \ref{thm: equivalence}. A clear drawback to this approach is the necessity of handling morphisms and maps at the same level.
We have solved the problem by extending the hom-functor through an appropriate profunctor (Definition \ref{def: semibiproduct}). Other possible solutions would consider maps as an extra structure in higher dimensions \cite{Brown,NMF.15,NMF.20a-in} or as imaginary morphisms \cite{BZ,MontoliRodeloLinden}.
Developing further a categorical framework in which to study semibiproducts seems desirable due to several important cases occurring in different settings. For example, semibiproduct extensions can be studied in the context of preordered monoids \cite{NMF.20b} and preordered groups \cite{Preord}, where the maps $q$ and $s$ are required to be monotone maps rather than zero-preserving maps. The context of topological monoids \cite{Ganci} should also be worthwhile studying with $q$ and $s$ required to be continuous maps.

\section*{Acknowledgements}

This work was supported by Fundação para a Ciência e a Tecnologia (FCT UID-Multi-04044-2019), Centro2020 (PAMI -- ROTEIRO\-/0328\-/2013 -- 022158) and by the Polytechnic of Leiria through the projects CENTRO\--01\--0247: FEDER\--069665, FEDER\--069603, FEDER\--039958, FEDER\--039969, FEDER\--039863, FEDER\--024533.


\end{document}